\documentclass{amsart}
\usepackage{graphicx, enumitem, mathtools}
\newtheorem{theorem}{Theorem}[section]
\newtheorem{lemma}[theorem]{Lemma}

\newtheorem{proposition}[theorem]{Proposition}

\theoremstyle{definition}

\newtheorem{notation}[theorem]{Notation}
\theoremstyle{remark}
\newtheorem{remark}[theorem]{Remark}
\numberwithin{equation}{section}

\newcommand{\eps}{\varepsilon}

\newcommand{\pow}{\frac{1}{\alpha}}

 \DeclarePairedDelimiter{\ceil}{\lceil}{\rceil} 

\usepackage{color,graphicx,amssymb}
\usepackage{graphicx, enumitem, mathtools}
\allowdisplaybreaks
\begin{document} 
\title{Mixing rates and limit theorems for random intermittent maps}
\author{Wael Bahsoun$^{\dagger}$}
\address{Department of Mathematical Sciences, Loughborough University,
Loughborough, Leicestershire, LE11 3TU, UK}
\email{$\dagger$ W.Bahsoun@lboro.ac.uk}
\author{Christopher Bose$^{*}$}
\address{Department of Mathematics and Statistics, University of Victoria,
   PO BOX 3045 STN CSC, Victoria, B.C., V8W 3R4, Canada}
\email{$*$ cbose@uvic.ca}
\thanks{The second author is supported by a research grant from the National Sciences and Engineering Research Council of Canada.}
\subjclass{Primary 37A05, 37E05}
\date{\today}
\keywords{Interval maps with a neutral fixed point, intermittency, random dynamical systems, decay of correlations, Central Limit Theorem, stable laws}
\begin{abstract}
We study random transformations built from intermittent maps on the unit interval that share a common neutral fixed point.  We focus mainly on random selections of Pomeu-Manneville-type maps $T_\alpha$ using the full parameter range $0< \alpha < \infty$, in general.  We derive a number of results around a common theme that illustrates in detail how the constituent map that is fastest mixing (i.e.\ smallest $\alpha$) combined with details of the randomizing process, determines the asymptotic properties of the random transformation.   Our key result (Theorem \ref{thm:asymptotics}) establishes sharp estimates on the position of return time intervals for the \emph{quenched} dynamics.   The main applications of this estimate are to \textit{limit laws} (in particular, CLT and stable laws, depending on the parameters chosen in the range
$0 < \alpha < 1$) for the associated skew product; these are detailed in Theorem \ref{thm:limits}.  Since our estimates in Theorem \ref{thm:asymptotics} also hold for $1 \leq \alpha < \infty$ we study a piecewise affine version of our random transformations, prove existence of an infinite ($\sigma-$finite) invariant measure and study the corresponding correlation asymptotics.  To the best of our knowledge, this latter kind of result is completely new in the setting of random transformations. 
\end{abstract}
\maketitle
\pagestyle{myheadings}
\markboth{Mixing rates and limit theorems}{W.\, Bahsoun, C.\, Bose}

\section{Introduction}

In recent years, a lot of attention has been given to examples of nonuniformly expanding (or nonuniformly hyperbolic) maps with neutral fixed points. It is well known that such models can exhibit a range of nonstandard dynamical/probabilistic behavior; they may be mixing, but display subexponential decay of correlations for H\"older observables, for example. Limit theorems such as CLT and stable laws can be derived within various classes depending on the \emph{strength} of the intermittency around the fixed point.
  
The purpose of this paper is to investigate similar questions for \emph{random transformations} whose constituent maps are drawn from an appropriate nonuniformly expanding family.  In particular, we aim to understand how behavior of the random transformation depends on properties of the maps and the randomizing process.  A brief synopsis of our findings is as follows:  At the level of existence (or non-existence) of a finite invariant measure and the \emph{rate} of correlation decay for sufficiently regular observables, the random dynamics are completely determined by the map with fastest relaxation, independent of the randomization.  The same also holds for the dynamical CLT when the correlation decay is strong enough to be summable.   However, we find the randomizing process begins to play an explicit role at the next finer level of analysis, for example, in sharp correlation asymptotics for regular observables supported away from the fixed points, and in limit theorems taking the form of stable laws for the associated skew product. Overall, this analysis gives a coherent picture that is consistent with our intuition about how randomness interacts with the intermittency.  

We will work in the following concrete setting.  Let $(I,\mathfrak B(I),m)$ denote the measure space consisting of the unit interval $I=[0,1]$ with Borel $\sigma-$algebra $\mathfrak{B}(I)$ and $m=$ Lebesgue measure
on $\mathfrak{B}(I)$.
The first part of this paper will concentrate on randomized one-dimensional maps of Pomeau-Manneville type \cite{PM}. 
A well-known, simplified version of the PM maps is the family of so-called Liverani-Saussol-Vaienti maps \cite{LSV}. Such systems have attracted the attention of both mathematicians and physicists (see \cite{LS} for a recent work in this area).\\

To set our notation, given a parameter value
$0<\alpha< \infty$, define
$$T_{\alpha}(x)=\begin{cases}
       x(1+2^{\alpha}x^{\alpha}) \quad x\in[0,\frac{1}{2}]\\
       2x-1 \quad \quad \quad x\in(\frac{1}{2},1].
       \end{cases}
$$

When $\alpha=0$,  $T_\alpha$ is the doubling map. For $\alpha >0$, $x=0$ is a neutral fixed point for the map $T_\alpha$ which is consequently a nonuniformly expanding, piecewise $C^\infty$, monotone map of the interval (on two pieces).

It is well-known that $T_\alpha$ admits a finite ACIM with density 
$h_\alpha = O(x^{-\frac{1}{\alpha}})$ for $x$ near zero when $0<\alpha<1$ (see Liverani, Saussol, Vaienti \cite{LSV}, for example) and an infinite, $\sigma-$finite ACIM with similar asymptotic near zero when $1 \leq \alpha< \infty$ (see Pianigiani \cite{Pi} for this range).
In fact, the argument in \cite{LSV} shows that for $0<\alpha<1$, the density $h_\alpha$ is locally Lipschitz on $(0,1]$ as well as being continuous and integrable.

Now fix two parameters $0<\alpha < \beta < \infty$ and consider the {\bf random LSV transformation} defined as follows. 
 $$T=\{T_{\alpha}(x),T_{\beta}(x); p_{1}, p_{2}\}, \text{ where }$$
$p_1,p_2>0$ and $p_2=1-p_1.$ The random transformation $T$ maybe viewed as a Markov process with transition function 
$$\mathbb P(x,A)=p_{1}{\bf 1}_{A}(T_{\alpha}(x))+p_{2}{\bf 1}_{A}(T_{\beta}(x))$$ 
of a point $x\in I$ into a set $A\in \mathfrak B(I)$. The transition function induces an operator, $E_T$, acting on measures; i.e., if $\mu$ is a measure on $(I, \mathfrak B)$,
$$(E_T\mu)(A)=p_1\mu(T_{\alpha}^{-1}(A)) + p_2\mu(T_{\beta}^{-1}(A)).$$
A measure $\mu$ is said to be $T$-invariant if $$\mu=E_T\mu,$$
and $\mu$ is said to be an absolutely continuous invariant measure if $d\mu=f^*dm$, $f^*\geq 0$.
To study absolutely continuous invariant measures, we introduce the transfer operator (Perron-Frobenius) of
the random transformation $T$:
$$(P_{T}f)(x)=p_1P_{T_{\alpha}}\left(f\right)(x)+p_2P_{T_{\beta}}\left(f\right)(x),$$
where $P_{T_{\alpha}}, P_{T_{\beta}}$ are the transfer operators associated with the $T_{\alpha}, T_{\beta}$ respectively. Then it is a straight-forward computation to show that a measure $\mu=f^*\cdot m$ is an absolutely continuous $T-$invariant measure if
$$P_Tf^*=f^*.$$

\subsection{A skew product representation}

Define the skew product transformation $S(x,\omega):I\times I\to I\times I$ by 
\begin{equation}\label{skew}
S(x,\omega)=(T_{\alpha(\omega)},\varphi(\omega)),
\end{equation}
where
\begin{equation}\label{alpha}
\alpha(\omega)=\begin{cases}
       \alpha \quad ,\omega\in[0,p_1)\\
       \beta \quad ,\omega\in[p_1,1]
       \end{cases}
;\quad\quad
\varphi(\omega)=\begin{cases}
       \frac{\omega}{p_1} \quad \quad ,\omega\in[0,p_1)\\
       \frac{\omega-p_1}{p_2} \quad ,\omega\in[p_1,1]
       \end{cases}.
\end{equation}       
The skew product representation in \eqref{skew} is a version\footnote{The results obtained in Bahsoun, Bose and Quas \cite{BBQ} are valid for any class of measurable non-singular maps on $\mathbb R^q$, without any regularity assumptions. Moreover in  \cite{BBQ}, the probability distribution on the noise space is allowed to be place-dependent.} of the skew product representation which was studied in Bahsoun, Bose and Quas \cite{BBQ}. We denote the transfer operator associated with $S$ by $\mathcal L_S$: for $g\in L^1(I\times I)$ and measurable $A \subseteq I \times I$,
$$
  \int_{S^{-1}A} \, g\, d(m \times m)(x,\omega) = \int_A \, \mathcal L_S g \,
  d(m \times m)(x,\omega).
$$
Then a measure $\nu$, such that $d\nu=g^*d(m\times m)$ and $\int_{I\times I}g^*d(m\times m)=1$, is an absolutely continuous $S$-invariant probability measure if 
$$\mathcal L_S g^*=g^*.$$
In \cite{BBQ}, Theorem 5.2 it is shown that if  $g \in L^1(I\times I)$ and $\mathcal L_S g =\lambda g$ with $|\lambda|=1$, then 
$$g(x,\omega)=f(x)\cdot {\bf 1}(\omega)$$
and $P_T f = \lambda f$, 
 that is, $g$ depends only on the spatial coordinate $x$ and 
 as a function of $x$ only, is also an eigenfunction for $P_T$. Setting $\lambda =1$ we obtain
$\mathcal L_S g^*=g^*$ if and only if  $g^*(x, \omega) = f^*(x)$ with  $P_T f^* = f^*$.  Consequently there is a one to one correspondence between invariant densities for $S$ and invariant densities for $T$.   Moreover, dynamical properties such as ergodicity, number of ergodic components or weak-mixing, properties that are determined by peripheral eigenfunctions,  can be determined via either system.

Our skew product construction is similar to a model constructed by Gou\"ezel \cite{G2}, however in that paper, the skew product samples \emph{continuously} from the space of LSV maps, whereas we sample discretely. This allows us to simplify the analysis and extend the range of parameters in which we can complete the analysis, compared to \cite{G2}. A more detailed discusion and comparison between the two models can be found in Bahsoun, Bose and Duan \cite{BBD}.

\subsection{Inducing for the skew representation $S$}\label{sec:inducing}

The method of inducing (equivalently, Markov extensions or Young towers) gives a systematic way to study maps like $T_\alpha$ having localized singularities, for example, as detailed in Young \cite{Y}.  We will begin by doing essentially the same thing with our skew product $S$, inducing on the right half of the square $\Delta_0:=(1/2,1] \times [0,1]$.  

Set
$$T^n_\omega(x):=T_{\alpha(\varphi^{n-1}\omega)}\circ...\circ T_{\alpha(\varphi\omega)}\circ T_{\alpha(\omega)}(x).$$
Then
$$S^n(x,\omega)=(T^n_\omega(x),\varphi^{n}(\omega)).$$
Also, set
$$P^n_\omega:=p_{\alpha(\varphi^{n-1}\omega)}\times...\times p_{\alpha(\varphi\omega)}\times p_{\alpha(\omega)},$$
where $ p_{\alpha(\omega)}=p_1,$ for $\alpha(\omega)=\alpha$ and $ p_{\alpha(\omega)}=p_2,$ for $\alpha(\omega)=\beta.$ 
We define two sequences of random points $\{x_n(\omega)\}$ and $\{x'_n(\omega)\}$ in $[0,1]$ which will be used to construct the first return map of $S$ to $\Delta_0$. The points $x_n(\omega)$ lie in $(0, 1/2]$. Set
\begin{equation}\label{def_backorbit}
x_1(\omega)\equiv\frac{1}{2}\text{ and } x_n(\omega)=T^{-1}_{\alpha(\omega)}\mid_{[0,\frac{1}{2}]}[x_{n-1}(\varphi\omega)], n\geq 2.
\end{equation}
Observe that with this notation, 
$$S(x_n(\omega),\omega) 
= (T_{\alpha(\omega)}(x_n(\omega)), \varphi \omega) = (x_{n-1}(\varphi \omega), \varphi \omega). $$
The points $x'_n(\omega)$  lie in $(\frac{1}{2},1]$, defined by
\begin{equation}\label{def_backorbit_prime}
x'_0(\omega)\equiv 1,x'_1(\omega)\equiv\frac{3}{4} \text{ and } x'_n(\omega)=\frac{x_n(\varphi \omega)+1}{2},n\geq 2,
\end{equation}
that is, the $x'_n(\omega)$ are preimages of the $x_n(\varphi \omega)$ in  $(\frac{1}{2},1]$ under the right branch $2x-1.$

\subsection{First return map of $S$ to $\Delta_0$}
Let $R: \Delta_0\to \mathbb{Z}^+$ be the first return time function and $S^R : \Delta_0\to \Delta_0$ be the return map. For $n\geq 1$, set $I_n(\omega):=(x_{n+1}(\omega),x_{n}(\omega)]$ and $J_n(\omega):=(x'_{n}(\omega),x'_{n-1}(\omega)]$.
Observe that every point in $J_n(\omega)$ will return to $(\frac{1}{2},1]$ in $n$ steps 
under the random iteration $T^n_{\omega}$ as follows:
$$J_n(\omega)\rightarrow I_{n-1}(\varphi \omega)\rightarrow I_{n-2}(\varphi^2 \omega)\rightarrow...\rightarrow I_{1}(\varphi^{n-1}\omega)\rightarrow(\frac{1}{2},1].$$
Next, we partition $\Delta_0$ into subsets $\Delta_{0,i}$, $i=1,2,\dots$ where
$$\Delta_{0,i}:= \{ (x,\omega) ~|~ x \in J_i(\omega) \}$$
and then further partition each $\Delta_{0,i}$ into subsets $\Delta^j_{0,i}, \, j=1, 2, \dots 2^i$ according to the $2^i$ possible
values  of the
string $\alpha(\omega), \alpha(\varphi \omega),  \dots \alpha(\varphi^{i-1}\omega)$.
Defined this way,  $S^i$ maps each subset $\Delta^j_{0,i}$ bijectively to $\Delta_0$. 

For example, in the case $i=2,$ there are four sets $\Delta^j_{0,2}$ on which 
$R=2$ and such that $S^R$ maps each set bijectively to $\Delta_0$:
$$\Delta^j_{0,2}=\begin{cases}
       J_2(\omega)\times [0,p^2_1),     & \mbox{if\text{ } } j=1,\\
       J_2(\omega)\times [p^2_1,p_1),     & \mbox{if\text{ } } j=2,\\
       J_2(\omega)\times [p_1,p_1+p_1\cdot p_2),      & \mbox{if\text{ } } j=3,\\
       J_2(\omega)\times [p_1+p_1\cdot p_2,1),      & \mbox{if\text{ } } j=4.\\
       \end{cases}$$
       
To summarize, 
$$\Delta_{0,i}=\bigcup\limits_{j=1}\limits^{2^i}\Delta^j_{0,i}$$
and 
$$\Delta_0 = \bigcup\limits_{i=1}\limits^{\infty}\bigcup\limits_{j=1}\limits^{2^i}\Delta^j_{0,i},$$
where, for every $i$ and  $j=1,2,...,2^i,$
$$R\mid_{\Delta^j_{0,i}}=i.$$

For each $n$, the interval $J_n(\omega)$ depends on only the first $n$ coordinates in 
$\omega$ and moreover 
$$m \times m \{R=n\} = \sum_{j=1}^{2^n}\,P_{\omega_j}^n m(J_n(\omega_j)) = E_\omega(m(J_n(\omega))),$$
where $\omega_j$ ranges across the $2^n$ possible configurations $\omega$
with distinct values for the string $\alpha(\omega), \alpha(\varphi \omega),  \dots \alpha(\varphi^{n-1}\omega)$ and $E_\omega(\cdot)$ denotes expectation with respect to $\omega$.
Since $m(J_n(\omega)) = \frac{1}{2} m (I_{n-1}(\omega))$ we also obtain 
\begin{equation}\label{eqn:m(R)}
m\times m\{R > n\} = E_\omega (x_n^\prime(\omega) - 1/2) = \frac{1}{2} E_\omega(x_n(\omega)).
\end{equation}

Finally, we adopt the following (standard) notation throughout this paper. Given  sequences $a_n,~ (\textnormal{respectively }b_n)$ of nonnegative (respectively positive) real numbers, we write $a_n \asymp b_n$ if there is a constant $C \geq 1$ such that 
$C^{-1} b_n \leq  a_n \leq C b_n$,
and $a_n \sim b_n$ if $\lim \frac{a_n}{b_n} = 1$.

\subsection{Statement of the main result in this paper}

There is now a range of studies (including Young \cite{Y}, Zweim\"uller \cite{Z1, Z2}, Sarig \cite{S1},  Gou\"ezel \cite{G1} and Melbourne-Terhesiu \cite{MT}, for example)  that show how careful analysis of the asymptotics of $m\{R > n\}$ can reveal deep statistical properties of the underlaying map.  In our case, we are interested in the skew $S$ acting on the square.  The strength of our results, therefore, are likely to depend in a critical way, on the sharpness of estimates obtained on the measures of sets like $J_n(\omega)$ and $I_n(\omega)$.

For example, in \cite{Y} a key estimate for a single LSV-map  $T_\alpha$ reads as follows:  if $x_n$ is the sequence of points generated under the inverse of the leftmost branch of $T_\alpha$, such that $T_\alpha x_{n+1} = x_n$ and $x_1 = 1/2$, then there exists $c>0$ such that 
 $c^{-1} n^{-\pow} \leq x_n \leq c n^{-\pow}$. 
If we introduce the notation $x_n(\alpha):= x_n$ for this sequence of deterministic points, we can
record this observation as  
\begin{equation}\label{eqn:lsy}
x_n(\alpha) \asymp n^{-\pow}
\end{equation}
which upper bounds the size of return-time sets and is sufficient for establishing existence of the invariant density $h_\alpha$ and bounds on rate of correlation decay when 
$0< \alpha <1$ as detailed in  \cite{Y}.  In fact, the analysis in \cite{Y} actually proves more: define $c_n(\alpha): = n^\pow x_n(\alpha)$.  Then  $x_n(\alpha) = c_n(\alpha)n^{-\pow}$ with 
$ \lim_n c_n(\alpha) = \frac{1}{2} \alpha^{-\pow}:= c(\alpha) $. That is, in our notation
\begin{equation}\label{eqn:c(alpha)}
x_n(\alpha) \sim \frac{1}{2 \alpha^\pow}n^{-\pow}= c(\alpha)n^{-\pow}
\end{equation}
This sharper estimate is key for analysis of limit theorems for maps like $T_\alpha$. See, for example, Melbourne and Terhesiu \cite{MT} and  Gou\"ezel \cite{G3}. 

Moving to similar estimates on our skew product $S$,
the following rough estimate is obtained as Lemma 4.4 in Bahsoun, Bose and Duan \cite{BBD} as a first step in their analysis:  For all $\omega \in [0,1]$ and $n \geq 1$ 

\begin{equation}\label{eqn:crude_bounds}
x_n(\alpha) \leq x_n(\omega) \leq x_n(\beta), 
\end{equation}
where $x_n(\beta)$ denotes the sequence of deterministic points for $T_\beta$.
The main result in this paper is a much sharpened estimate on the location of $x_n(\omega)$ compared to Equation (\ref{eqn:crude_bounds}). 

Keeping the bounds (\ref{eqn:crude_bounds}) in mind, and following the setup for equation (\ref{eqn:c(alpha)}), for each 
$n, \omega$ define $c_n(\omega) := n^\pow x_n(\omega)$ (so that $x_n(\omega) = c_n(\omega)n^{-\pow}$).   We now state the main result in our paper: 

\begin{theorem}\label{thm:asymptotics} Let $0 < \alpha < \beta < \infty$.   For almost every $\omega$ in $[0,1]$ 
we have $ \lim_n  c_n(\omega)= \frac{1}{2} (\alpha p_1)^{-\frac{1}{\alpha}}= c(\alpha)p_1^{-\frac{1}{\alpha}}.$  That is
\begin{equation}\label{eqn:main}
x_n(\omega) \sim c(\alpha) p_1^{-\pow} n^{-\pow}.
\end{equation}
Moreover, $E_\omega(|c_n(\omega) - c(\alpha)p_1^{-\frac{1}{\alpha}}|) \rightarrow 0$, in other words, convergence of $n^\pow x_n$ to $c(\alpha)p_1^{-\pow}$ also holds in the $L^1-$norm. 
\end{theorem}

In the terminology of random dynamical systems, this is a \emph{quenched limit theorem} (ie: almost everywhere) as opposed to \emph{annealed} (averaged over $\omega$).  In general, quenched results are harder to obtain than annealed ones. Examples of other quenched  
limit theorems can be found in 
Ayyer, Liverani and Stenlund \cite{ALS}.

The significance of Theorem \ref{thm:asymptotics} is what it implies for asymptotics of the random system. 
The main applications of this theorem will appear in Sections 3 and 4 where we derive limit theorems for the skew product $S$ and study asymptotics for infinite measure preserving systems, respectively. However, to illustrate the flavour of our results in a simple context, we close this section by revisiting (and sharpening\footnote{Statement (3) in the current Theorem \ref{thm:old} is essentially proved in \cite{BBD}. The exact asymptotics of $m \times m\{ R >n\}$ however, are new, and the precise decay of correlations in (4) for functions supported away from the line $x=0$ are also new.})  the main conclusion from 
\cite{BBD} that shows one way in which the fast system ($T_\alpha$) dominates the asymptotic behavior of the skew:  

\begin{theorem}\label{thm:old} Let $0<\alpha<\beta < 1$ and $S$ be as defined in \eqref{skew}. Then   $m \times m\{R>n\} \sim \frac{1}{2} c(\alpha)p_1^{-\pow} n^{-\pow} 
=\frac{1}{4} (\alpha p_1)^{-\pow} n^{-\pow}$.  Moreover, 
\begin{enumerate}
\item $S$ admits a unique absolutely continuous invariant probability measure $\nu$ with 
density $d\nu = h d(m \times m)$ where $h$ is Lipschitz on compact subsets of $(0,1] \times [0,1]$;
\item $(S, \nu)$ is mixing;
 \item for $\phi\in L^{\infty}(I\times I, m\times m)$ and $\psi$ a H\"older continuous function on $I\times I$ $$|Cor(\phi,\psi)|= O(n^{1-\frac{1}{\alpha}}),$$
where 
$$Cor(\phi,\psi)=\int\phi\circ S^n\cdot\psi d\nu-\int\phi d\nu\int\psi d\nu.$$\end{enumerate}
With more assumptions on the observables $\phi$ and $\psi$ we obtain the following stronger estimate:
\begin{enumerate}[resume]
\item  If $\phi\in L^{\infty}(I\times I, m\times m)$ and $\psi$ Lipschitz on $I \times I$, 
$\int \phi\, d\nu \neq 0$ and  $\int \psi \, d\nu \neq 0$ with both $\phi$ and $\psi$ identically $0$ in an open strip containing the line $x=0$, then 
  $$Cor(\phi,\psi) 
 \sim \frac{1}{4}E_\omega (h(\frac{1}{2}, \omega))(\alpha p_1)^{-\pow}\bigl( \pow -1\bigr)^{-1}n^{1-\pow}\int \phi \, d\nu \int \psi \, d\nu.$$
 \end{enumerate}

\end{theorem}

\begin{proof}
The enumerated statements (1 -- 3) follow by identical arguments to those in \cite{BBD} once the claimed asymptotics on $m \times m\{R>n\}$ are derived. The latter are easily established since by Equation (\ref{eqn:m(R)}) we know
$m \times m \{R> n\} = \frac{1}{2}E_\omega(x_n(\omega)) 
= \frac{1}{2} n^{-\pow} E_\omega(c_n(\omega))$,
while  Theorem \ref{thm:asymptotics} implies
$E_\omega(c_n(\omega)) \rightarrow c(\alpha)p_1^{-\frac{1}{\alpha}} \in (0, \infty)$.
It follows that $m \times m \{ R > n \} \sim \frac{1}{2}c(\alpha)p_1^{-\pow}n^{-\pow}$.
The fact that the density $h$ is Lipschitz on compact subsets of $(0,1] \times I$ is proved in Lemma \ref{Lip} in Section 3 of this paper. 

To establish (4) we first assume that $\phi, \psi$ are supported on $\Delta_0$ with the 
stated regularity.  By Theorem 6.3 in Gou\"ezel \cite{G1} we have
$$Cor(\phi,\psi) \sim \sum_{k>n}  \nu\{ R >k\} \int \phi \, d\nu \int \psi \, d\nu.$$
The invariant measure $d\nu = h dm\times dm$, with $h$ Lipschitz on $\Delta_0$. 
This leads to the following estimate (see Lemma \ref{lem:ass_tail} in Section 3) on the measure of return time sets:
$$\nu\{ R>k\} \sim \frac{1}{4}E_\omega (h(\frac{1}{2}, \omega))(\alpha p_1)^{-\pow}k^{-\pow}.$$
Summing over $k > n$ gives the result.   

Now using the argument from Gou\"ezel  \cite{G1} Section 7 in our setting we can extend the support of $\phi, \psi$ to $\Delta_N := \{ (x,\omega) :  x_N(\omega) \leq x \leq 1\}$, with the same asymptotic return times. For sufficiently large $N$ this picks up the support of $\phi$ and $\psi$.
\end{proof}

\bigskip 
The rest of this paper is organized as follows. 
In Section 2 we prove Theorem \ref{thm:asymptotics}.  The computation depends on a classical  result of Hoeffding \cite{Ho} that gives exponental decay of large deviations for sums of bounded, independent random variables. 

In Section 3 we apply our estimates to derive central limit theorems and stable laws for the Birkhoff sums 
\begin{equation}\label{eqn:Bsum}
S_nf (x, \omega) := \sum_{k=0}^{n-1}f( S^k(x, \omega)),
\end{equation}
where $S$ is the skew product and parameters $0<\alpha < \beta <1$ pass through various ranges leading to quantitatively different asymptotics.   Precise results are itemized in Theorem \ref{thm:limits}. 
Finally, in Section 4 we discuss a piecewise affine version of the skew product system. When $1 \leq  \alpha < \beta$  a natural analogue of Theorem \ref{thm:asymptotics} applies, and the invariant measure is bound to be infinite ($\sigma-$finite).  We investigate correlation asymptotics for this case. To the best of our knowledge this is the first detailed analysis of asymptotics for a random system with an infinite invariant measure. 

\section{Proof of Theorem \ref{thm:asymptotics}} 

We begin with a basic calculus estimate. 

\begin{lemma}\label{lem:taylor}  Let $0< \alpha < \infty$ and $0 \leq x \leq 1$. Then 

$$1 - \alpha x \leq [1 + x]^{-\alpha} \leq  1- \alpha x + \frac{\alpha (1+ \alpha)}{2} x^2 $$
\end{lemma} 

\begin{proof}
Elementary. 
\end{proof}

Recall our notation $x_n(\alpha)$ (resp.\ $x_n(\beta)$) for the sequence of points generated by first branch inverse of the deterministic map $T_\alpha$ (resp.\ 
 $T_\beta$). 
Recall also the basic estimates in Equations (\ref{eqn:lsy}), (\ref{eqn:c(alpha)}) and 
(\ref{eqn:crude_bounds}).
Our goal is to obtain sharp decay estimates on 
$n [x_n(\omega)]^\alpha$ and for that we will need the following classical large deviations estimate.

\begin{proposition}\label{prop:large_deviations} (Hoeffding \cite{Ho}, Theorem 1)
Suppose that $X_k = X_k(\omega), \, k=1, 2, \dots n$ are independent random variables, uniformly bounded such that $0\leq X_k \leq 1$.  Let 
$\bar X_n = n^{-1} \sum_{k=1}^n X_k$ and $E(\bar X_n)= n^{-1} \sum_{k=1}^n E(X_k)$.  Then for every $t >0$ we have
$$\mathbb{P}\{ |\bar X_n - E(\bar X_n)| \geq t \} \leq  \exp(-2nt^2)$$
\end{proposition}

We proceed by a sequence of lemmas.   
  
\begin{lemma}\label{lem:limsup}
There is a set $G_1 \subseteq [0,1]$ of full measure such that for every $\omega \in G_1$ we have 
$$\limsup n^{\frac{1}{\alpha}} x_n(\omega) \leq 
 [\alpha 2^\alpha p_1]^{-\frac{1}{\alpha}} = c(\alpha)p_1^{-\frac{1}{\alpha}}.$$
\end{lemma}

\begin{proof}
We begin with the standard expression derived directly from the definition of 
$T_{\alpha(\omega)}$:
\begin{equation}\label{eqn:basic}
\frac{1}{[x_{n-1}(\varphi \omega)]^\alpha} = \frac{1}{[x_{n}(\omega)]^\alpha}
[1 +  [2x_n(\omega)]^{\alpha(\omega)}]^{-\alpha}.
\end{equation}
Using the upper bound contained in the right hand side of Lemma \ref{lem:taylor}, with $x=[2x_n(\omega)]^{\alpha(\omega)}$ and 
reordering terms we obtain
\begin{equation*}
\frac{1}{[x_{n}(\omega)]^\alpha} - \frac{1}{[x_{n-1}(\varphi \omega)]^\alpha} \geq 
\alpha 2^\alpha[2x_n(\omega)]^{\alpha(\omega)-\alpha} - 
\frac{\alpha (1+ \alpha )}{2} 2^\alpha[2x_n(\omega)]^{2\alpha(\omega) - \alpha}.
\end{equation*}
Applying this inequality along the sequence $x_k(\varphi^{n-k}\omega)$ for $k=2$ through $n$, and keeping in mind that $x_1(\omega) = \frac{1}{2}$ for every 
$\omega$ gives the basic inequality 
\begin{equation*}\begin{split}
\frac{1}{[x_{n}(\omega)]^\alpha} \geq 2^\alpha &+ \alpha 2^\alpha \{
\sum_{k=2}^n[2x_k(\varphi^{n-k}\omega)]^{\alpha(\varphi^{n-k}\omega) - \alpha} \\
&- \frac{1+ \alpha }{2} \sum_{k=2}^n [2x_k(\varphi^{n-k}\omega)]^{2\alpha(\varphi^{n-k}\omega) - \alpha}\}.
\end{split}
\end{equation*}

Next, we use the estimate contained in Equation (\ref{eqn:crude_bounds}),  the notation from Equation (\ref{eqn:c(alpha)}) and division by $n$ to obtain 
\begin{equation*}\begin{split}
\frac{1}{n[x_{n}(\omega)]^\alpha} \geq \frac{2^\alpha}{n} &+ 
\frac{n-1}{n}\alpha 2^\alpha \biggl\{
\frac{1}{n-1}\sum_{k=2}^n\left[\frac{2c_k(\alpha)}{k^{\frac{1}{\alpha}}}\right]^{\alpha(\varphi^{n-k}\omega) - \alpha}\\ 
&- \frac{1+ \alpha }{2} \frac{1}{n-1} \sum_{k=2}^n \left[\frac{2c_k(\beta)}{k^{\frac{1}{\beta}}}\right]^{2\alpha(\varphi^{n-k}\omega) - \alpha}\biggr\}.
\end{split}
\end{equation*}
Now consider the quantity
\begin{equation} \label{eqn:A_n}
\begin{split}
A_n(\omega) &:= 
\frac{1}{n-1}\sum_{k=2}^n\left[\frac{2c_k(\alpha)}{k^{\frac{1}{\alpha}}}\right]^{\alpha(\varphi^{n-k}\omega) - \alpha} \\
&- \frac{1+ \alpha }{2} \frac{1}{n-1} \sum_{k=2}^n \left[\frac{2c_k(\beta)}{k^{\frac{1}{\beta}}}\right]^{2\alpha(\varphi^{n-k}\omega) - \alpha}
\end{split}
\end{equation}
We estimate each sum in $A_n$ independently using the large deviations estimate detailed in 
Proposition \ref{prop:large_deviations}. 

For the first sum in Equation (\ref{eqn:A_n}), using the substitution 
$$X_k(\omega) = \left[\frac{2c_k(\alpha)}{k^{\frac{1}{\alpha}}}\right]^{\alpha(\varphi^{n-k}\omega) 
- \alpha}$$
from which we compute
$$E_\omega(X_k) = p_1 + p_2\left[\frac{2c_k(\alpha)}{k^{\frac{1}{\alpha}}}\right]^{\beta-\alpha},$$
and using Proposition \ref{prop:large_deviations} and a positive value $t=t_n>0$ we obtain 
\begin{equation}\begin{split}
\mathbb P\{ \biggl|\frac{1}{n-1}\sum_{k=2}^n\left[\frac{2c_k(\alpha)}{k^{\frac{1}{\alpha}}}\right]^{\alpha(\varphi^{n-k}\omega) - \alpha} &- \frac{1}{n-1}\sum_{k=2}^n ( p_1 + p_2\left[\frac{2c_k(\alpha)}{k^{\frac{1}{\alpha}}}\right]^{\beta-\alpha})\biggr|
\geq t_n\} \\ &\leq \exp(-2(n-1)t_n^2).
\end{split}
\end{equation}
If we choose $t_n \downarrow 0$ such that\footnote{$t_n = n^{-1/3}$ does the job, for example.} $\sum_n\exp(-2(n-1)t_n^2 ) <\infty$ then by 
Borel-Cantelli, keeping in mind that $c_k(\alpha)$ is bounded, and 
$$\frac{1}{n-1}\sum_{k=2}^n (p_1 + p_2\left[\frac{2c_k(\alpha)}{k^{\frac{1}{\alpha}}}\right]^{\beta-\alpha}) = p_1 + O(n^{1 -\beta/\alpha}),
$$
we conclude that 
\begin{equation}\label{eqn:first_term}
\frac{1}{n-1}\sum_{k=2}^n\left[\frac{2c_k(\alpha)}{k^{\frac{1}{\alpha}}}\right]^{\alpha(\varphi^{n-k}\omega) - \alpha} \rightarrow p_1
\end{equation}
for almost every $\omega \in [0,1]$.

We can follow a similar argument for the second term in Equation (\ref{eqn:A_n}). 
This time however, 
$$ E_\omega(\frac{1+ \alpha }{2} \frac{1}{n-1} \sum_{k=2}^n \left[\frac{2c_k(\beta)}{k^{\frac{1}{\beta}}}\right]^{2\alpha(\varphi^{n-k}\omega) - \alpha})= O(n^{-\gamma}),
$$
where $\gamma = \min\{ \alpha/\beta, 2-\alpha/\beta\} = \alpha/\beta >0$ since 
$\alpha < \beta$. Therefore we conclude that for almost every $\omega$
\begin{equation}\label{eqn:second_term}
\frac{1+ \alpha }{2} \frac{1}{n-1} \sum_{k=2}^n [\frac{2c_k(\beta)}{k^{\frac{1}{\beta}}}]^{2\alpha(\varphi^{n-k}\omega) - \alpha} \rightarrow 0.
\end{equation}
Combining Equations (\ref{eqn:first_term}), and (\ref{eqn:second_term}) 
 shows that $A_n \rightarrow p_1$ almost everywhere. 

It follows that almost surely (w.r.t.\ $\omega$) $\liminf_n \frac{1}{n [x_n(\omega)]^\alpha} 
\geq \alpha 2^\alpha p_1$. The statement of the lemma follows. 
\end{proof}

\begin{lemma}\label{lem:liminf}
There is a set $G_2 \subseteq [0,1]$ of full measure such that 
for every $\omega \in G_2$ we have
$$\liminf n^{\frac{1}{\alpha}} x_n(\omega) \geq
 [\alpha 2^\alpha p_1]^{-\frac{1}{\alpha}} = c(\alpha)p_1^{-\frac{1}{\alpha}}.$$
\end{lemma}

\begin{proof}
Let $G_1$ be a set of full measure in $\omega$ for which convergence is obtained in Lemma \ref{lem:limsup}. In particular, for every $\omega \in G_1$ there exists an $N=N(\omega)$ 
such that for all $n> N(\omega)$ we have 
$$x_n(\omega) \leq \frac{c(\alpha) p_1^{-\frac{1}{\alpha}}+1}{n^{\frac{1}{\alpha}}}.$$

Now, starting with Equation (\ref{eqn:basic}), using the lower bound in 
Lemma \ref{lem:taylor}, dividing by $n$, and assuming $\lfloor \sqrt{n}\rfloor \geq N(\omega)$ we get the following expression:
\begin{equation*}\begin{split}
\frac{1}{n [x_n(\omega)]^\alpha}  &\leq \frac{1}{n} 2^\alpha
+ \alpha 2^\alpha \frac{1 }{n} \sum_{k=2}^n
 [2 x_k(\varphi^{n-k} \omega)]^{\alpha(\varphi^{n-k} \omega) - \alpha}\\
& \leq \frac{1}{n} 2^\alpha
+ \alpha 2^\alpha \frac{ \lfloor \sqrt{n}\rfloor -1 }{n} \frac{1}{\lfloor \sqrt{n}\rfloor -1 } \sum_{k=2}^{\lfloor \sqrt{n}\rfloor}
 [2 x_k(\varphi^{n-k} \omega)]^{\alpha(\varphi^{n-k} \omega) - \alpha} \\
&+\frac{n- \lfloor \sqrt{n}\rfloor}{n} \frac{1}{n- \lfloor \sqrt{n}\rfloor}\sum_{\lfloor \sqrt{n}\rfloor +1}^n 
 \left[2 \frac{c(\alpha) p_1^{-\frac{1}{\alpha}}+1}{n^{\frac{1}{\alpha}}}\right]^{\alpha(\varphi^{n-k} \omega) - \alpha}
 \end{split}
 \end{equation*}
 Now define, for any $\omega \in [0,1]$
 \begin{equation}\label{eqn:A_n-prime}
 A_n^\prime(\omega) := \frac{1}{n-\lfloor \sqrt{n}\rfloor} \sum_{k=\lfloor \sqrt{n}\rfloor+1}^n
 \left[2 \frac{c(\alpha) p_1^{-\frac{1}{\alpha}}+1}{n^{\frac{1}{\alpha}}}\right]^{\alpha(\varphi^{n-k} \omega) - \alpha}
 \end{equation}
 Once again, application of the large deviation estimates in Proposition \ref{prop:large_deviations} combined with the direct calculation 
$E_{\omega}(A_n^\prime) = p_1  + O(n^{1 - \beta/\alpha})$ 
shows that $A_n^\prime \rightarrow p_1$ for almost every $\omega$ in a set $G_2$ of full measure. 

Finally, fix an arbitrary $\omega \in G_1 \cap G_2$. Provided $n$ is large enough such 
that $\lfloor \sqrt{n} \rfloor \geq N(\omega)$ we estimate
\begin{equation*}
\frac{1}{n [x_n(\omega)]^\alpha}   \leq \frac{1}{n}2^\alpha 
+ \alpha 2^\alpha \frac{\lfloor \sqrt{n}\rfloor}{n}  + \alpha 2^\alpha \frac{n-\lfloor \sqrt{n}\rfloor}{n} A_n^\prime(\omega).
\end{equation*}

The right hand side of this expression converges to $\alpha 2^\alpha p_1$.
It follows that for all  $\omega \in G_1 \cap G_2$,
$\liminf n [x_n(\omega)]^\alpha \geq  \alpha 2^\alpha p_1$.  The lemma now follows by taking roots.

\end{proof}

Lemmas \ref{lem:limsup} and \ref{lem:liminf} together give the almost sure convergence claimed in Theorem \ref{thm:asymptotics}.  To obtain the $L^1$ convergence we first observe

\begin{lemma}\label{lem:limsup_exp}
$$ \limsup_{n\to\infty} E_\omega(n^{\frac{1}{\alpha}} x_n(\omega)) \leq c(\alpha)p_1^{-\frac{1}{\alpha}}.$$
\end{lemma}
\begin{proof}
Let $p_0\in(0,p_1)$. By Proposition 4.1 of \cite{BBD} we have
\begin{equation}\label{eqn:expectation_lower}
E_\omega(n^{\frac{1}{\alpha}} x_n(\omega)) \leq n^{\frac{1}{\alpha}}x_{\lfloor{p_0n}\rfloor}(\alpha)+n^{\frac{1}{\alpha}}\text{exp}(-2n(p_1-p_0)^2).
\end{equation}
Let $\eps>0$ be arbitrary. Since $x_n(\alpha) = c_n(\alpha)n^{-\frac{1}{n}}$, with 
$\inf c_n(\alpha) >0$ (see Equation (\ref{eqn:c(alpha)})), for $n$ large enough, \eqref{eqn:expectation_lower} implies
\begin{equation}\label{eqn:expectation}
 E_\omega(n^{\frac{1}{\alpha}} x_n(\omega)) \leq (1+\eps)n^{\frac{1}{\alpha}}x_{\lfloor{p_0n}\rfloor}(\alpha).
\end{equation}
We also know from Equation (\ref{eqn:c(alpha)}) that
\begin{equation}\label{e2}
x_{\lfloor{p_0n}\rfloor}(\alpha)= c_{\lfloor{p_0n}\rfloor}(\alpha){\lfloor{p_0n}\rfloor}^{-\frac{1}{\alpha}},
\end{equation}
with $\lim_{n\to\infty} c_{\lfloor{p_0n}\rfloor}(\alpha)=c(\alpha)$. 
Consequently, by \eqref{eqn:expectation} and \eqref{e2}, we have
$$ E_\omega(n^{\frac{1}{\alpha}} x_n(\omega)) \leq (1+\eps)n^{\frac{1}{\alpha}}c_{\lfloor{p_0n}\rfloor}(\alpha){\lfloor{p_0n}\rfloor}^{-\frac{1}{\alpha}}\le (1+\eps)n^{\frac{1}{\alpha}}c_{\lfloor{p_0n}\rfloor}(\alpha){p_0}^{-\frac{1}{\alpha}}(n-\frac{1}{p_0})^{-\frac{1}{\alpha}}.$$
Thus,
\begin{equation}\label{e3}
 \limsup_{n\to\infty}E_\omega(n^{\frac{1}{\alpha}} x_n(\omega))\leq (1+\eps)c(\alpha)p_0^{-\frac{1}{\alpha}}.
\end{equation}
Since $p_0$ can be taken arbitrarily close to $p_1$, and $\eps>0$ is arbitrary, \eqref{e3} implies $ \limsup_{n\to\infty} E_\omega(n^{\frac{1}{\alpha}} x_n(\omega)) \leq c(\alpha)p_1^{-\frac{1}{\alpha}}$.
\end{proof}
Finally, Lemmas \ref{lem:limsup} and \ref{lem:liminf} combined with Lemma \ref{lem:limsup_exp} give the required $L^1$ convergence due to the following elementary result.

\begin{lemma} \label{lem:L^1_conv} (Also Lemma 4.3 of Gou\"ezel \cite{G2}) 
Let $f_n$ be a sequence of integrable functions on a probability space, with $f_n \geq 0$ and 
$f_n \rightarrow f$ almost everywhere. Suppose that $E( f ) < \infty$ and 
 $\limsup E(f_n ) \leq E( f )$. Then $E(|f_n -f|) \rightarrow 0$ (i.e. $f_n \rightarrow f$ in $L^1$ norm).
 \end{lemma}
 
 We include the proof for completeness. 
 
 \begin{proof}
 Set $g_n := \min\{ f_n, f\} = \frac{1}{2}\{ f_n + f - | f_n - f|\}$.   Then $0 \leq g_n \leq f$ and
 $g_n \rightarrow f$ almost everywhere, so by dominated convergence $E(|g_n - f|) \rightarrow 0$.  From this, we also see $E(g_n)  \rightarrow E(f)$.
 
 Now $|g_n - f_n| = f_n - g_n \geq 0$ so by the first part
\begin{equation*}\begin{split}
\limsup E(|g_n - f_n|) &= \limsup E(f_n) - \liminf E(g_n)\\
&\leq E(f) - E(f)\\
&=0\end{split}
\end{equation*}

It follows that $\lim E(|g_n - f_n|) =0$ which completes the proof.

  \end{proof}
  
\section{Application to limit theorems}

In this section we apply Theorem \ref{thm:asymptotics} to establish limit theorems for the Birkhoff sums $S_n f(x, \omega)$ in (\ref{eqn:Bsum}) when $f: I \times I \rightarrow \mathbb R$ is  
H\"older-continous and $\int f d\nu =0$.  Here, $\nu = h\, dm\times dm$ is the absolutely continuous invariant measure for $S$.

We first observe that for $0<\alpha < \beta < 1$ the induced map $(S^R, \Delta_0)$ is \emph{Gibbs-Markov} for the return-time partition 
$\Delta_{0, i}$ (see Aaronson \cite{Aa1} or Aaronson and Denker \cite{AD} for a definition of Gibbs-Markov).  The required expansion and distortion estimates are derived in \cite{BBD} with respect to the metric $d(z_1, z_2) := \theta^{s(z_1, z_2)}$, for a suitable constant $0 < \theta < 1$, 
 $z_1= (x_1, \omega_1),~ z_2= (x_2, \omega_2)$  and $ s(z_1, z_2)$ being the usual 
 \emph{separation time} of two points in $\Delta_0$ with respect to the return partition. 

Next, we establish local Lipschitz regularity for the density $h$ away from the fixed point. 

\begin{lemma}\label{Lip}
For $0<\alpha<\beta<1$, $h$ has a version that is Lipschitz on any compact subset of $(0,1]\times [0,1]$. 
\end{lemma}
\begin{proof}
By Theorem 5.2 of \cite{BBQ}, we know that the unique invariant density of $S$ is of the form $h=g\times{\bf 1}$ where $P_Tg=g$,  $P_T$ being the transfer operator associated with the random map $T$.  Thus, it is enough to show that $g$ is Lipschitz on compact subsets of $(0,1]$. We prove this fact by studying the action of $P_T$  on a suitable cone. Let $\mathcal{B}$ denote the set of integrable and $C^1$ functions on $(0,1]$. 
For $a>0,$ define a cone $\mathcal C_a$ by
$$\mathcal{C}_{a}=\{f\in\mathcal{B}\mid f\geq 0, f \text{ decreasing}, \int\limits_{0}\limits^{x}fd\lambda\leq ax^{1-\beta}\int_{0}^{1} f\}.$$
Let $a_*= \frac{4}{1-\beta}$. For $a\ge a_*$, it is well known that under the action of $P_{T_{\beta}}$, the transfer operator associated with the map $T_{\beta}$, $\mathcal{C}_{a}$ is invariant. Moreover, since $\alpha<\beta$ the same is true for $P_{T_{\alpha}}$. Since $P_T$ is a convex combination of $P_{T_{\alpha}}$ and $P_{T_{\beta}}$, we conclude that $\mathcal{C}_{a}$ is invariant under the action of $P_T$. Consequently,  $P_T$ has a fixed point in $\mathcal{C}_{a}$ (an invariant density), which we denoted by $g$. Note that $g$ is Lipschitz on any compact subset of $(0,1]$ by the properties of $\mathcal{C}_{a}$.  
\end{proof}

From this point on we assume $h$ is the (unique) Lipschitz version assured by this lemma since this plays a key role in some of the estimates to follow.

Set $A:= \frac{1}{2}c(\alpha)p_1^{-\pow}E_\omega(h(\frac{1}{2}, \omega))$.  
Let $\mathcal N(0, \sigma^2)$ denote the normal distribution with mean zero and variance $\sigma^2$.

\begin{theorem}\label{thm:limits}  Let $0< \alpha < \beta <1$. 
Let $f: I\times I \rightarrow \mathbb R$ be a H\"older continuous function satisfying
$\int f \, d\nu =0$.
Set $c:= E_\omega(f(0, \omega))$. Then
\begin{enumerate}
\item If $\alpha < \frac{1}{2}$, there exists $\sigma^2 \geq 0$  such that 
$$\frac{1}{\sqrt n} S_n f \rightarrow \mathcal{N}(0, \sigma^2).$$
\item  If $\frac{1}{2} \leq \alpha < 1$ and $c=0$, suppose there exists a
$\gamma > \frac{\beta}{\alpha}(\alpha - \frac{1}{2})$ such that 
$|f(x, \omega) - f(0, \omega)| \leq C_f x^\gamma$.  Then there exists
$\sigma^2 \geq 0$ such that 
$$\frac{1}{\sqrt n} S_n f \rightarrow \mathcal{N}(0, \sigma^2).$$
\item
If $\alpha = \frac{1}{2}$ and $c \neq 0$ then $S_nf / \sqrt{c^2 A n \ln n} \rightarrow \mathcal{N}(0,1)$.
\item 
If $\frac{1}{2} < \alpha < 1$ and $c \neq 0$ then $S_nf / n^\alpha  \rightarrow Z$
where the random variable $Z$ has characteristic function given by 
$$E(\exp(itZ)) = \exp\{ -A |c|^\pow \Gamma(1 - \pow) \cos (\pi/2 \alpha)
|t|^\pow (1-i \,\textnormal{sgn}(ct)\tan (\pi/2 \alpha))\}.$$
\end{enumerate}
\end{theorem}

The proof depends on a number of careful estimates using Theorem \ref{thm:asymptotics} that can be proved in an analogous way to corresponding calculations in Gou\"ezel  \cite{G2}. 

\begin{lemma}\label{lem:ass_tail}
We have $\nu(R>n)\sim n^{-1/\alpha}A$. 
\end{lemma}

\begin{notation}  For $f : [0,1] \times [0,1] \rightarrow \mathbb R$ and $(x,\omega) \in \Delta_0$ define
$$f_{\Delta_0} (x, \omega) := \sum_{k=0}^{R(x, \omega)-1} f(S^k(x, \omega)).$$
\end{notation}
\begin{lemma}\label{lem:central}
Let $f$ be H\"older on $[0,1]\times[0,1]$. If $0<\alpha<1/2$, then $f_{\Delta_0} \in L^2(\Delta_0, d\nu)$. 
\end{lemma}

\begin{lemma}\label{lem:5.4}
Suppose $f: X \rightarrow \mathbb R$ is H\"older continuous with 
$E_\omega(f(0,\omega))=0$.  Suppose there are constants $0< \gamma < \beta$ and  $C_f < \infty$ such that, uniformly in $\omega$, for all 
$x \in [0,1]$ 
$$ | f(x, \omega) - f(0, \omega) | \leq C_f x^\gamma.$$
Let $1 \leq p < \min\{ 2/ \alpha, 1/ \alpha(1-\frac{\gamma}{\beta})\}$.  Then 
$f_{\Delta_0} \in L^p(\Delta_0, \nu)$. 

\end{lemma}

We are going to work with a family of metrics based on separation time. For 
$2 \leq \lambda< \infty$, the expansion constant for $S^R$ and any $0<\theta <1$ define 
$$d_{\lambda^{-\theta}}(z_1, z_2) := \lambda^{-\theta s(z_1, z_2)}.$$
It follows that with respect to $d_{\lambda^{-\theta}}$, the expansion constant 
of $S^R$ is at least $\lambda^{\theta} > 1$, and that $S^R$ is Gibbs-Markov for 
this metric and the partition $\Delta_{0,i}^s$.

Let $0<\theta <1$ be the H\"older exponent of $f$.    For each $n, s$ let 
$Df_{\Delta_0}(\Delta_{0,{n}}^s)$ denote the Lipschitz constant of $f_{\Delta_0}$ restricted to the subset 
$\Delta_{0,n}^s$ and computed with respect to the metric $d_{\lambda^{-\theta}}$.

\begin{lemma}\label{lem:fifty}
$$ \sum_{n, s} \nu( \Delta_{0,{n}}^s)Df_{\Delta_0}(\Delta_{0,{n}}^s) 
\leq C\sum_n \nu\{ R =n\} n < \infty.$$
\end{lemma}

\subsection{Proof of Theorem \ref{thm:limits}}
 
We are going to use Theorem 3.1 of Gou\"ezel \cite{G2}.  The basic finite expectation  condition (Equation  (18)  in \cite{G2}) is given in our setting by Lemma \ref{lem:fifty}.

Assume first that $\alpha < \frac{1}{2}$ .  Then by Lemma \ref{lem:central} we know that
$f_{\Delta_0} \in L^2(\Delta_0, \nu)$.  Also the return time function $R \in L^2$ since 
$R=g_{\Delta_0}$ with $g \equiv 1$, to which Lemma \ref{lem:central} also applies.
This is the setting of the first case of Theorem 3.1, so we obtain the central limit theorem in (1). 

Next we consider $1/2 \leq \alpha <1$ and $c=0$. Now Lemma \ref{lem:5.4} with conditions given on $\gamma$ shows that $f_{\Delta_0} \in L^2$ and the 
estimate in Lemma \ref{lem:ass_tail} shows that 
$$\nu\{ R> n\} = n^{-\pow}A(n) \sim n^{-\pow}A.$$
For $z\in (0,\infty)$ we get $\nu\{ R> z\} = \ceil{z}^{-\pow}A(\ceil{z})$.
Set $\Lambda(z):=  \bigl(\ceil{z}/z \bigr )^{-\pow}A(\ceil{z})$.  Note that 
$z \rightarrow \Lambda(z)$ is slowly varying, $\Lambda(z) \sim A$ and  $\nu\{ R> z\} = z^{-\pow}\Lambda(z)$.
The second sub-condition in the first case of Theorem 3.1 \cite{G2} are therefore satisfied with $L:= \Lambda$ and we again get a central limit theorem in (2). 

The last two cases require a more detailed estimate.
Assume $\frac{1}{2} \leq \alpha <1$.   Set $g \equiv c$ and note that 
on $g_{\Delta_0}(x,\omega) = cn \iff R(x,\omega)=n$, so $\nu\{ |g_{\Delta_0}| >z\} \sim
 |c|^{\pow}z^{-\pow}\Lambda(z)$ according to the previous calculation.    The function $j=f-g$ has the same regularity as 
$f$ and satisfies $E_\omega(j(0,\omega))=0. $
We next show that $\nu\{|j_{\Delta_0}| >z \} = o(z^{-\pow})$.   Applying 
Lemma \ref{lem:5.4} to $j$, we obtain $p> \pow$ such that $j_{\Delta_0} \in L^p( \Delta_0)$.
It follows that 
$$ \nu\{ |j_{\Delta_0}| > z \} \leq \int \bigl (| j_{\Delta_0}|/z\bigr )^p \, d\nu = Cz^{-p} =o\bigl ( z^{-\pow} \bigr).$$
The elementary decomposition
$$ \{ g_{\Delta_0} > z(1+\epsilon)\} \cap \{ | j_{\Delta_0}| \leq \epsilon z \} 
\subseteq \{ f_{\Delta_0} > z\} \subseteq \{ g_{\Delta_0} > z(1-\epsilon)\} \cup\{ | j_{\Delta_0}| >\epsilon z\}$$
implies  
\begin{equation}\label{eqn:tricky}\begin{split}
\nu \{ g_{\Delta_0} > z(1+\epsilon)\} - \nu \{ | j_{\Delta_0}| > \epsilon z\} 
&\leq \nu \{ f_{\Delta_0} > z\} \\ 
&\leq \nu\{ g_{\Delta_0} > z(1-\epsilon)\} + \nu \{ | j_{\Delta_0}| >\epsilon z\}.\\
\end{split}
\end{equation}

Now consider the case $\alpha = \frac{1}{2}$ in (3). Assume $c>0$ and use
 the asymptotic estimates on the upper and lower bounds in Equation (\ref{eqn:tricky}) to obtain
$\nu\{ f_{\Delta_0} >z \} \sim z^{-2} (c^2 A + o(1))$.   On the other hand, if $c<0$ then 
$g<0$ and $\{ g_{\Delta_0} > z(1\pm \epsilon)\} = \emptyset$, so 
$\nu\{ f_{\Delta_0} >z \} \leq \nu\{ |j_{\Delta_0}| > \epsilon z \} \sim z^{-2} o(1)$. Combining these two estimates yields 
$\nu\{| f_{\Delta_0} |>z \} \sim z^{-2} (c^2 A + o(1)):= z^{-2}l(z)$, independent of the sign of $c$.  The only difference is that for $c>0$ the tail distribution is heavy for postive $z$ while for $c<0$ it is heavy for negative $z$.  We have already established
that $\nu\{ R>z\} \sim z^{-2}A = z^{-2}(1/c^2)l(z)$.  We can now apply the third case of  Theorem 3.1 \cite{G2} with $L(z) = 2c^2A \int_1^z \frac{1}{u} \,du = 2c^2A \ln z$ (unbounded and slowly varying)
and $B_n := \sqrt{c^2 A n \ln n}$, whereby $nL(B_n) \sim B_n^2$ as required. 

For the final case, when $\frac{1}{2} < \alpha <1$ we return to the estimate in Equation (\ref{eqn:tricky}) and again, we first assume $c>0$, so that $g > 0$.  For 
$z >0$ the asymptotic estimates yield
$$\nu\{ f_{\Delta_0} >z \} \sim z^{-\pow} (c^\pow A + o(1))$$
while we have already established that
$\nu\{ R>z\} \sim z^{-\pow}A$.
On the other hand, 
$$\nu\{ f_{\Delta_0} < -z \} \leq \nu \{ |j_{\Delta_0}| > 
\epsilon z\}  = o\bigl( z^{-\pow} \bigr).$$
We can therefore apply the last case in Theorem 3.1, \cite{G2}, setting $c_1 = c^{\pow}A$, $c_2=0$,  $c_3 = A$,  $L \equiv 1$ and $B_n := n^\alpha$.  
In the case $c<0$ we simply exchange the values of $c_1$ and $c_2$.  Putting this together, the theorem gives an asymptotic stable law with characteristic function 
$$\exp\bigl[ -c^\pow A \Gamma(1-\pow)\cos\bigl(\frac{\pi}{2 \alpha}\bigr) |t|^\pow
\bigl(1-i \,\textnormal{sgn}(c)\,\textnormal{sgn}(t)\tan\bigl(\frac{\pi}{2 \alpha}\bigr)\bigr)\bigr],$$
which is case (4) in our theorem. 
This completes the proof.

 \section{Application to correlation asymptotics for infinite measure preserving random systems}
In this section we use arguments analogous to Theorem \ref{thm:asymptotics} to study the asymptotics of the transfer operator associated with the skew product of a piecewise affine version of the random model discussed in the pervious sections. In particular, we will consider\footnote{We consider a linearized version because for general random LSV transformations we were able to prove bounded distortion only when $0<\alpha\le\beta\le 1$. See \cite{BBD} for the result and a discussion on distortion.} \begin{equation*}
\tilde T_{\alpha(\omega)}(x)=\begin{cases}
       \frac{x_{n-2}(\varphi \omega)-x_{n-1}(\varphi\omega)}{x_{n-1}(\omega)-x_{n}(\omega)} (x- x_{n}(\omega)) + x_{n-1}(\varphi\omega),\\ 
\hskip 4cm \textnormal{ for }x \in(x_{n}(\omega),x_{n-1}(\omega)], n= 1, 2, \dots\\
       2x-1, \quad \quad \quad \hskip 1.6cm\textnormal{ for } x\in(\frac{1}{2},1].
       \end{cases}
\end{equation*}
Let 
\begin{equation}\label{GW}
S(x,w):=(\tilde T_{\alpha(\omega)}(x),\varphi\omega)
\end{equation}
denote the associated skew product. As in Subsection \ref{sec:inducing}, we induce $S$ on $\Delta_0$. Thus, by our theorem, we can apply Theorem 1.4 of\footnote{For more general observables, one can use the result of \cite{MT}. However, one would lose uniform convergence for all parameters. See \cite{MT, G3} for a discussion.} \cite{G3} to obtain asymptotics of $\mathcal L_S^n$. In particular, we obtain the following theorem.
\begin{theorem}\label{infinite}
Let $S$ be the skew product defined using \eqref{GW} with $1\le\alpha<\beta<\infty$. The following hold:
\begin{enumerate}
\item $S$ admits a unique (subject to the normalizing condition $\nu( \Delta_0)=1$) absolutely continuous invariant infinite ($\sigma$-finite) measure $\nu$;
\item Let $\mathcal L_{\nu}$ be  the transfer operator associated with $S$ with respect to the invariant measure $\nu$. For $\alpha >1$, let $f$ be a Lipschitz function supported on $\Delta_0$. Then 
$$\lim_{n\to\infty}|| n^{1-\frac{1}{\alpha}}{\bf 1}_{\Delta_0}\mathcal L_{\nu}^n f-c\int_{\Delta_0}f||_{\infty}=0,$$
where $c$ is a constant independent of $f$. In particular,  if $g\in L^1(\Delta_0)$, we have 
$$\lim_{n\to\infty}n^{1-\frac{1}{\alpha}}\int_{\Delta_0}f\cdot g\circ S^n= c\int_{\Delta_0}f\int_{\Delta_0}g.$$
\item For $\alpha=1$, and $f$ a Lipschitz function supported on $\Delta_0$, we obtain the same results as in (2) with normalizing sequence $\ln n$ instead of $n^{1-\pow}$. 
\end{enumerate}
\end{theorem}
\begin{proof}
We first notice that induced skew product $S^R:\Delta_0\to\Delta_0$ is piecewise affine and onto. In particular, it satisfies the assumptions of Aaronson-Denker \cite{AD}. Thus, it has a unique absolutely continuous invariant probability measure $\nu_{\Delta_0}$ whose density $h_{\Delta_0} \in\mathcal B$, where $\mathcal B$ is the space of Lipschitz functions on $\Delta_0$, the system $(S^R, \nu_{\Delta_0})$ is mixing and the associated transfer operator $\mathcal L_{S^R}$, with respect to $\nu_{\Delta_0}$, has a spectral gap on $\mathcal B$. The $S$-invariant measure, $\nu$, is defined using $\nu_{\Delta_0}$. The fact that the measure $\nu$ is infinite ($\sigma$-finite) follows from Theorem \ref{thm:asymptotics} since $\alpha\ge1$. Since $\nu|_{\Delta_0} = \nu_{\Delta_0}$ the normalization $\nu(\Delta_0)=1$ is automatically satisfied. To prove (2), we apply Gou\"ezel \cite{G3} Theorem 1.4. For $f\in\mathcal B$ define 
\begin{equation}\label{renewal} 
R_nf :=1_{\Delta_0} \mathcal L_{\nu}^n(1_{\{R=n\}}f),
\end{equation}
where
$\mathcal L_{\nu}$ is the transfer operator associated with $S$ with respect to the invariant measure $\nu$. Using \eqref{renewal} and $\mathcal L_{S^R}$, we get
$$\mathcal L_{S^R}(f)=\sum_{n\ge 1}\mathcal L_{\nu}^n(1_{\{R=n\}}f).$$
The spectral properties of $\mathcal L_{S^R}$ imply that $(R_n)_{n\ge1}$ is an aperiodic renewal sequence of operators (see Sarig \cite{S1} for the definition). We still need to check:\\
\begin{itemize}
\item $\nu(\{R>n\})\sim n^{-\frac{1}{\alpha}}l(n)$, where $l$ is slowly varying function;\\
\item $\exists\, C>0$ such that $||R_n||_{\text{Lip}}\le C n^{-\frac{1}{\alpha}-1}.$
\end{itemize}
The first condition follows from Lemma \ref{lem:ass_tail}. For the second one, using (8) on page 649 of \cite{S1}, $\exists C>0$ such that
$$||R_n||_{\text{Lip}}\le C \nu(\{R=n\}).$$
Thus, Lemma \ref{lem:ass_tail} completes the proof of (2). For (3) the proof is essentially the same as in (2) but using part (a) of Theorem 1.1 in \cite{MT}.
\end{proof}
\begin{remark}
When $0<\alpha <1$ and $\alpha < \beta < \infty$ we can use the arguments from Theorem \ref{thm:old} to prove that the skew product of linearized 
random transformation has a finite invariant measure over the full range of $0<\beta<\infty$, even though the $T_{\beta}$ maps have only infinite absolutely continuous invariant measures when $1 \leq \beta < \infty.$ The required bounded distortion condition is automatically satisfied for the skew product \eqref{GW}.  Asymptotic estimates as in Theorem \ref{thm:asymptotics} lead to decay of correlation results as in Theorem \ref{thm:old} and limit laws as in Theorem \ref{thm:limits}.
\end{remark}
\noindent {\bf{Acknowledgment.}} The authors thank Ian Melbourne for bringing the piecewise affine maps in reference \cite{GW} to their attention. The authors are in debt to Henk Bruin for raising a question about the linearized random system in section 4. His question prompted us to correct the definition of the skew product in equation \eqref{GW}.  
\bibliographystyle{amsplain}

\end{document}